%% file: bare_jrnl.tex
\documentclass[journal,twoside,web]{IEEEtran}
\pdfoutput=1
\usepackage{balance}

\usepackage{graphics} 
\usepackage{epsfig} 
\usepackage{amsmath} 
\usepackage{amssymb}  
\usepackage{amsthm}
\usepackage{bbm}
\usepackage{stmaryrd}
\usepackage{mathabx}
\usepackage{color}
\usepackage{tikz}
\usepackage{cite}
\usepackage{lipsum}
\usepackage{todonotes}
\usepackage{cite}
\usepackage{relsize}

\usetikzlibrary{patterns,calc,angles,quotes,positioning,datavisualization}
\tikzset{
  treenode/.style = {shape=rectangle, rounded corners,
                     draw, align=center,
                     top color=white, bottom color=blue!20},
  root/.style     = {treenode, font=\Large, bottom color=red!30},
  env/.style      = {treenode, font=\ttfamily\normalsize},
  dummy/.style    = {circle,draw,minimum size = 0.6cm,pattern=crosshatch, pattern color = black!50},
  graph/.style    = {circle,draw,minimum size = 0.6cm,pattern=dots, pattern color = black!25},
  action/.style   = {circle,draw,minimum size = 0.6cm,pattern=crosshatch, pattern color = black!25},
}

\newtheorem{lemma}{Lemma}
\newtheorem{theorem}{Theorem}

\newtheorem{definition}{Definition}

\newtheorem{remark}{Remark}
\newtheorem{proposition}{Proposition}

\newcommand{\lap}[0]{{\mathcal{L}}}

\newcommand{\graph}[0]{{\mathcal{G}}}
\newcommand{\Graph}[0]{{\mathcal{G}}}

\newcommand{\Tr}{\mathrm{Tr}}

\newcommand{\Nodes}{\mathcal{N}}

\newcommand{\Edges}{\mathcal{E}}
\newcommand{\Weights}{\mathcal{W}}

\newcommand{\Htwo}{{\mathcal{H}_2}}
\newcommand{\blk}{{\mathrm{Blk}}}

\usepackage{mathtools}
\usepackage{soul}
\usepackage{tikz}
\usepackage{color}
\usepackage{tcolorbox}
\usepackage{float}
\usepackage{subfig}
\usepackage[noend, noline]{algorithm2e}
\newlength\mylen
\newcommand\myinput[1]{%
  \settowidth\mylen{\KwIn{}}%
  \setlength\hangindent{\mylen}%
  \hspace*{\mylen}#1\\}

\makeatletter
\newcommand{\removelatexerror}{\let\@latex@error\@gobble}
\makeatother
\makeatletter
\renewcommand{\@algocf@capt@plain}{above}
\makeatother

\usepackage{tikzscale}
\mathtoolsset{showonlyrefs}
\usepackage{textcomp}

\begin{document}
\title{$\Htwo$ Performance of Series-Parallel Networks:\\A Compositional Perspective}

\author{Mathias Hudoba de Badyn, and
        Mehran Mesbahi 
\thanks{Manuscript received December 11, 2018; revised October 16, 2019.
  The authors are with the William E. Boeing Department of Aeronautics and Astronautics at the University of Washington, Seattle, WA 98195 USA
  MHdB is also with the Automatic Control Laboratory, ETH, Z\"{u}rich, Switzerland.
  e-mails: \texttt{\{hudomath,mesbahi\}@uw.edu}. 
The research of the authors was supported by to following agencies: NSERC (CGSD2-502554-2017), the ARL/ARO (W911NF-13-1-0340), and the AFOSR (FA9550-16-1-0022).
  A preliminary version of this work was presented at the 2019 American Control Conference~\cite{Hudobadebadyn2019}.}%
}

\markboth{Transactions on Automatic Control,~Vol.~xx, No.~x, August~20xx}%
{Hudoba de Badyn \MakeLowercase{\textit{et al.}}: $\mathcal{H}_2$ on Series-Parallel Networks}

\maketitle

\begin{abstract}
We examine the $\mathcal{H}_2$ norm of matrix-weighted leader-follower consensus on series-parallel networks.
By using an extension of electrical network theory on matrix-valued resistances, voltages and currents, we show that the computation of the $\Htwo$ norm can be performed efficiently by decomposing the network into atomic elements and composition rules.
Lastly, we examine the problem of efficiently adapting the matrix-valued edge weights to optimize the $\Htwo$ norm of the network.
\end{abstract}


\IEEEpeerreviewmaketitle

\input{./1-intro.tex}
\input{./2-math-prelim.tex}
\input{./4-performance.tex}


\bibliographystyle{IEEEtran}
\bibliography{references.bib}

\input{./5-appendix.tex}

\end{document}

%% file: 1-intro.tex
\section{Introduction}

Networked systems are ubiquitous in many disciplines, including robotic swarms~\cite{Li2017}, animal flocking \cite{vicsek1995},  oscillator networks \cite{Matheny2019}.
A widely analyzed algorithm in networked systems is \emph{consensus}, a distributed information-sharing protocol over a network.
Consensus has diverse applications in control and estimation, such as multi-agent systems~\cite{Chen2013a,Saber2003}, distributed Kalman filtering~\cite{Olfati-Saber2005},  and swarm deployment~\cite{Hudobadebadyn2018}.
System-theoretic properties of this algorithm have been extensively examined in the literature, typically using the \emph{graph Laplacian}~\cite{Mesbahi2010}.
This setup  allows one to draw connections between systems aspects of distributed algorithms and the graph structure, e.g., highlighting the role of symmetries in inducing uncontrollable modes~\cite{Rahmani2009a}.
Various performance measures for networks have been examined in the literature, such as entropy and other so-called spectral measures~\cite{HudobaDeBadyn2015a,Siami2017,Mesbahi2010}, as well as disturbance rejection~\cite{Siami2014a,Chapman2013a,Bamieh2012}.
One can use such measures to  modify the topology of the network to adapt to antagonistic influences~\cite{Chapman2013a}.

Most of the existing literature on consensus considers the case where each node in the network has a \emph{scalar} state.
Recently, several extensions to the case of \emph{vector}-valued node states have been proposed; such networks have \emph{matrix-weighted edges}.
Properties of such networks, and their use in formation control were considered in~\cite{Trinh2018a}.
Further applications of matrix-valued weighted graphs are found in the control of coupled oscillators~\cite{Tuna2016}, opinion dynamics~\cite{Parsegov2017}, spacecraft formation control~\cite{Ramirez2009}, and distributed estimation~\cite{Lee2016a,Barooah2008}.

The notions of performance and control of consensus rely on the underlying topology of the network.
As such, it is desirable to devise algorithms for system theoretic analysis and synthesis that are scalable
and modular.
One approach is to take smaller, atomic elements and build large-scale graphs from them; one can then use graph-growing operations that preserve properties such as controllability~\cite{Hudobadebadyn2016,chapman2014cart}.

In this paper, we approach the $\mathcal{H}_2$ performance problem on leader-follower consensus by utilizing the paradigms of matrix-weighted resistor networks and in particular \emph{series-parallel networks}.
Such networks can be decomposed into smaller networks via simple operations in sublinear time~\cite{Eppstein1992}.
Many NP-hard problems on general classes of graphs become linear on series-parallel graphs, such as finding maximum matchings induced subgraphs and independent sets, and the maximum disjoint triangle problem~\cite{Takamizawa1986}.
This has been exploited in a number of other disciplines.
Efficient algorithms have been derived for the Quadratic Assignment Problem~\cite{cela2013quadratic}, the discrete time-cost tradeoff problem~\cite{De1997}, and resource allocation by dynamic programming~\cite{Elmaghraby1993}.

The contributions of the paper are as follows.
Using the matrix-valued resistance extension of electrical networks, we show that the analogous notion of effective resistance is related to the $\Htwo$ norm on a leader-follower consensus network with vector states.
By exploiting the decomposability of series-parallel networks, we present a way of computing the $\Htwo$ norm of a leader-follower consensus network in best-case $\mathcal{O}(k^\omega|\mathcal{R}|\log|\Nodes|)$ (worst case $\mathcal{O}(k^\omega|\mathcal{R}||\Nodes|)$) complexity, where $|\mathcal{R}|, |\Nodes|$ are the number of leaders and followers, respectively, and $\mathcal{O}(k^\omega)$ is the complexity of inverting a $k\times k$ symmetric positive-definite matrix; the current best lower bound for $\omega$ is 2.3728639~\cite{Gall2014}.
We also provide a gradient descent method for adaptively re-weighting the network to optimize $\Htwo$ performance that utilizes computations of similar complexity by again using the decomposition of series-parallel networks.
The outline of the paper is as follows: 
in \S\ref{sec:preliminaries}, the notation, preliminaries and problem setup are presented.
Our main results on $\Htwo$ computation/adaptive re-weighting are in \S\ref{sec:syst-theor-comp}, with conclusions in \S\ref{sec:conclusion}.


%% file: 2-math-prelim.tex
\section{Preliminaries \& Setup}
\label{sec:preliminaries}
\subsection{Mathematical Preliminaries}
For a matrix $A$, we respectively denote its inverse, pseudoinverse, transpose and conjugate transpose as $A^{-1}, A^\dag, A^\intercal, A^*$.
A \emph{graph} $\Graph$ is a triple of sets $(\Nodes,\Edges,\Weights)$, where $\Nodes$ is a set of \emph{nodes}, $\Edges\subseteq \Nodes^2$ is a set of \emph{edges} denoting pairwise connections between nodes, and $\Weights = \{W_{uv} \in \mathcal{S}_{++}^{k}~:~\{u,v\}\in\Edges\}$ is a set of matrix-valued \emph{weights} on the edges, where $ \mathcal{S}_{++}^{k\times k}$ denotes the set of $k\times k$ symmetric positive-definite matrices with real entries.
A graph is called \emph{directed} if the edge $\{i,j\} \neq \{j,i\}$, and for edge $\{i,j\}$ the node $i$ ($j$) is called the \emph{head} (\emph{tail}) of $\{i,j\}$.
$\Graph$ is \emph{undirected} if for all $\{i,j\}\in\Edges$, $\{i,j\} = \{j,i\}$.
A graph is called \emph{simple} if it is undirected, and the weights are all 1 (or identity).
{To \emph{identify the nodes} $\{i_1,i_2,...,i_r\} \triangleq \mathcal{I}\subseteq \mathcal{N}$ is to define an equivalence relation `$\sim$' such that $i_k\sim i_j$ for all $i_k,i_j \in \mathcal{I}$.
Then, the resulting graph $\mathcal{G}' = (\mathcal{N}', \mathcal{E}')$ is one whose nodes are the equivalence classes induced by `$\sim$':
$  
    \mathcal{N}' = \big\{[i] : i \in \mathcal{N} \big\} = \big\{ \{j \in \mathcal{N} : i \sim j\} : i \in \mathcal{N}\big\},
$
and the edge set $\{[i],[j]\} \in \mathcal{E}'$ for all $\{i,j\} \in \mathcal{E}$.
Given two disjoint graphs $\Graph_1,\Graph_2$, one may combine them through this operation in the following way.
Suppose $s_1\in\Nodes_1,s_2\in\Nodes_2$.
Then we write $\Graph \leftarrow s_1\sim s_2$ to indicate that $\Graph$ is obtained from $\Graph_1,\Graph_2$ by identifying $s_1$ and $s_2$.}

The \emph{incidence matrix} $E$ is a $|\Nodes|\times|\Edges|$ matrix, where each column of $E$ corresponding to an edge $\{i,j\}$ is denoted by $a_{ij}$.
For each edge $l:=\{i,j\}$, where $i$ is the tail and $j$ is the head, $E_{il} = 1$ and $E_{jl}=-1$.
If $\graph$ is undirected, by convention we write that $E_{il} = 1$ and $E_{jl}=-1$ for $i>j$.
Since we are dealing with matrix-valued weights, for defining the graph Laplacian below, we need the matrices $\mathbf{E}$ $ \triangleq E\otimes I_k$ and $\mathbf{A}_{ij} \triangleq a_{ij}\otimes I_k$, where $I_k$ is the $k\times k$ identity matrix, and $\otimes$ denotes the Kronecker product. %
The \emph{weight matrix} $\bf W$ is a $k|\Edges|\times k|\Edges|$ blockwise diagonal matrix containing the weights $W_{ij}$ of each edge $e$.
The \emph{graph Laplacian} $\lap$ of an undirected graph $\graph$  can be defined by the incidence and weight matrix as 
$  \lap \triangleq \mathbf{E}\mathbf{W}\mathbf{E}^{\intercal}  = \sum_{ij\in\Edges} \mathbf{A}_{ij}W_{ij} \mathbf{A}_{ij}^{\intercal} $.

A \emph{tree} is a connected graph with no cycles, and a \emph{leaf} is designated as a node of degree 1.
A \emph{binary tree} is a tree where one node is designated as the \emph{root}, and all nodes of $\mathcal{T}$ are either leaves or \emph{parents}.
Each parent in a binary tree has one parent and at most two \emph{children}, except the root which has no parent.
The \emph{height} $h$ of a binary tree is the length of the longest path from the root to a leaf.
A \emph{complete} (sometimes called \emph{full}) binary tree is one where each node has either zero or two children.
Finally, the \emph{parallel addition} of two symmetric matrices $A,B$ is defined as $A:B \triangleq A(A+B)^\dag B$.

\subsection{Problem Setup}
\label{sec:problem-statement}

{
In this paper, we examine the $\mathcal{H}_2$ performance of leader-follower consensus problem on matrix-valued weighted series-parallel networks.
Consider a connected weighted graph $\mathcal{G} = (\Nodes,\Edges,\Weights)$ with the Laplacian $\lap$.
Each node $i$ has a state $x_i\in \mathbb{R}^k$.
Denote a set of \emph{leaders} $\mathcal{R}\subset \Nodes$ and a set of \emph{followers} $\Nodes\setminus\mathcal{R}$.
Suppose that one is able to take over the state $x_i\in \mathbb{R}^k$ of a leader, and thereby exert control over the followers.
Further, suppose that each leader is connected by an edge with identity weight to a \emph{unique} node in $\Nodes\setminus\mathcal{R}$, that collectively will be called the \emph{source} nodes and designated as the set $R$ (see Figure~\ref{fig:leaderFollower} for a schematic of the setup).
Then, using $\mathbf{B} := B\otimes I_k$, the graph Laplacian of $\graph$ can be partitioned as,
\begin{figure}
  \centering
  \raisebox{-.5\height}{%
    \includegraphics[width=0.35\columnwidth]{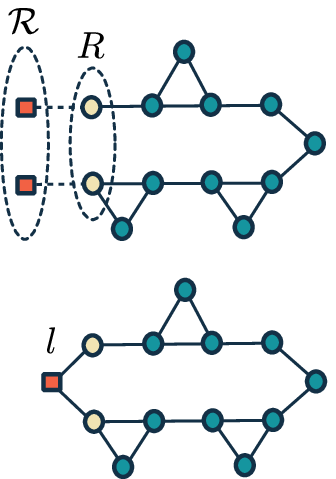}%
  }\qquad
  \raisebox{-.5\height}{%
    \includegraphics[width=0.35\columnwidth]{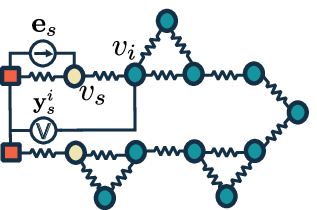}%
  }
  \caption{\textbf{Top:} leader-follower network setup. \textbf{Right:} electrical ``grounding'' of the leader set $\mathcal{R}$, and current vector $\mathbf{e}_s = e_s\otimes I_k$ injected into the network via node $v_s$; $\mathbf{y}_s^i$ is the voltage dropped from $v_i$ to $\mathcal{R}$. \textbf{Bottom:} Identification of grounded leader set into a single node.}
  \label{fig:leaderFollower}
\end{figure}
\begin{align}
  \lap = \left[
  \begin{array}{c|c}
    \mathcal{L}_{\mathcal{R}}  & -\mathbf{B} \\\hline
    -{\mathbf{B}}^{\intercal} & \mathcal{L}_{\mathcal{G}(\Nodes\setminus\mathcal{R})} + \sum_{i\in {R}} \mathbf{e}_iW_i\mathbf{e}_i^{\intercal}
  \end{array}\right],\label{eq:4}
\end{align}
where $\mathbf{e}_i\triangleq e_i \otimes I_k$.
The control matrix is given by $\mathbf{B}^{\intercal}  = [e_{i_1}~\dots~e_{i_m}]$ where $R = \{i_1,\dots,i_m\}$ are the nodes attached to leaders.
The graph Laplacian is written as,
\begin{align}
 \mathcal{L}_{\mathcal{G}(\Nodes\setminus\mathcal{R})} &= \mathbf{E}_{\Nodes\setminus\cal R}\mathbf{W}\mathbf{E}^{\intercal} _{\Nodes\setminus\cal R}  = \sum_{\{i,j\}\in\Edges(\Nodes\setminus\mathcal{R})} \mathbf{A}_{ij}W_{ij} \mathbf{A}_{ij}^{\intercal} ,
\end{align}
where $\mathbf{E}_{\Nodes\setminus\cal R} = E_{\Nodes\setminus\cal R} \otimes I_k$ and $E_{\Nodes\setminus\cal R}$ is the result of removing the rows from $E$ corresponding to the nodes in the leader set $\cal R$.
The matrix $\mathbf{W}=\mathrm{Blkdiag}[W_e]$ denotes the matrix consisting of the positive-definite weights $W_e$ on the block-diagonal.
The \emph{corresponding leader-follower consensus dynamics} are now given by,
\begin{align}
  \dot{\mathbf{x}} &= - \left(  \mathcal{L}_{\mathcal{G}(\Nodes\setminus\mathcal{R})} + \sum_{i\in {R}} \mathbf{e}_iW_i\mathbf{e}_i^{\intercal}  \right)\mathbf{x} + \mathbf{B}^{\intercal} u, \label{eq:A}
\end{align}
where $\mathbf{x}$ is the vector containing the stacked states of the nodes, and $u$ is the stacked vector of the leader node states--the control inputs to the followers.
We note that the \emph{Dirichlet Laplacian} $A(W) \triangleq (\mathcal{L}_{\mathcal{G}(\Nodes\setminus\mathcal{R})} + \sum_{i\in {R}} \mathbf{e}_iW_i\mathbf{e}_i^{\intercal} )$ is positive definite if $\mathcal{G}$ is connected.
}

{
  For a linear system with measurement $y=Cx$ and transfer function $G(s) = C(sI-A)^{-1}B$, the $\mathcal{H}_2$ norm is defined as, 
\begin{align}
  \|G(s)\|_2^2 = \dfrac{1}{2\pi} \int_{-\infty}^\infty \Tr \left[G(j\omega)^* G(j\omega) \right] d\omega.
\end{align}
The $\mathcal{H}_2$ norm is a measure of the input-output energy excitation of the system.
If we want to examine the holistic response of the system, we can measure all of its internal states; as such, we subsequently assume that $C = I$.
In the scalar case of \eqref{eq:A} (i.e., $k=1$), the $\Htwo$ norm of this system is given by \cite{Chapman2015},
\begin{align}
 \left(\Htwo^{\graph, B}\right)^2 = \dfrac{1}{2} \Tr \left(B^{\intercal}  A(W)^{-1} B \right).
\end{align}
In the case of matrix-valued edge weights, we have the following proposition. 
\begin{proposition}
  Consider the leader-follower consensus dynamics in~\eqref{eq:A}
  for positive integer $k$.
  Suppose that all nodes are observed, i.e. $y = \mathbf{x}$.
  Then the $\Htwo$ norm is given by 
  \begin{align}
    \left( \Htwo^{\graph, \mathbf{B}} \right)^2 = \dfrac{1}{2} \Tr \left( \mathbf{B}^{\intercal}  A(W)^{-1} \mathbf{B} \right).
  \end{align}
\end{proposition}
\begin{proof}
  The $\Htwo$ norm is given by $\Tr(\mathbf{B}^{\intercal}  P \mathbf{B})$, where $P$ satisfies the Lyapunov equation 
  \begin{align}
    -A(W) P - P A(W)^{\intercal}  + I = 0.
  \end{align}
  The ansatz $P = \frac{1}{2} A(W)^{-1}$ yields the solution.
\end{proof}
}

{
The contributions of the paper are the following.
First, we use an electrical perspective on the leader-follower consensus dynamics, as well as a decomposition of series-parallel networks, to compute the $\mathcal{H}_2$ performance on a class of two-terminal series-parallel graphs.
In particular, we show that when adding networks in series and parallel, the $\mathcal{H}_2$ norm follows a similar composition procedure as adding resistors in series and parallel.
Secondly, we provide an adaptive procedure to re-weight the network to minimize the $\mathcal{H}_2$ norm--this procedure computes the `power' dissipated across the network as part of a gradient update, and so can also be computed efficiently using a decomposition of series-parallel networks.
}

{
This extends the work in~\cite{Hudobadebadyn2019} by considering \emph{vector-valued} node states, and therefore \emph{matrix-valued} edge weights.
We show that by using a generalization of electrical network theory with \emph{matrix-valued resistors}, such as in~\cite{Barooah2007}, the electrical interpretation of the $\mathcal{H}_2$ performance of leader-follower consensus holds, and that the computations in the scalar-state case can naturally be extended to analogous, but non-trivial, electrical computations in the vector-state case.
}

\subsection{Electrical Network Models}
\label{sec:electr-netw-models}

One can view consensus through the lens of electrical network theory~\cite{Chapman2013a,Barooah2007}.
In the leader-follower setup of \S\ref{sec:problem-statement}, 
consider the graph $\graph = (\Nodes,\Edges,\Weights)$ with $W_e \in \mathcal{S}^n_{++}$.
The weight $W_e$ represents a \emph{matrix-valued resistor} on each edge $e:=ij$ with \emph{matrix-valued conductance $W_{ij}$}.
A \emph{generalized current} (see \cite{Barooah2007}) from node $u$ to $v$ with \emph{intensity} $\mathbf{i}\in\mathbb{R}^{k\times k}$ is an edge function $\mathbf{I}:\Edges \to \mathbb{R}^{k\times k}$ such that 
\begin{align}
  \sum_{\substack{\{k,l\}\in\Edges \\ k = p}} \mathbf{I}_{\{k,l\}} - \sum_{\substack{\{l,k\}\in\Edges \\ k = p}} \mathbf{I}_{\{l,k\}} = 
  \begin{cases}
    \mathbf{i} & p = u\\
    -\mathbf{i} & p = v\\
    \mathbf{0}_{k\times k} & \text{else}
  \end{cases}~~~\forall p \in \Nodes,
\end{align}
and there exists a node function $\mathbf{V}:\Nodes\to \mathbb{R}^{k\times k}$ satisfying
\begin{align}
  R^{\text{eff}}_{\{u,v\}} \mathbf{I}_{\{u,v\}} = \mathbf{V}_u - \mathbf{V}_v,~~~\forall \{u,v\}\in\Edges.
\end{align}
In this setting, the \emph{power} dissipated across an edge with a matrix weight $R_e^{\text{eff}}$ and current $\mathbf{I}_e$ is given by the inner product,
\begin{align}
  P_e(i_e) = \Tr\left(\mathbf{I}_e^{\intercal}  R_e^{\text{eff}} \mathbf{I}_e \right),
\end{align}
which reduces to the familiar formula $P = i^2 R$ when $k=1$.
\begin{proposition}
  Let $R_1$ and $R_2$ denote two-matrix valued resistances in $\mathcal{S}^k_{++}$, and consider the current $\mathbf{I}\in\mathbb{R}^{k\times k}$ across $R_1$ and $R_2$ in parallel.
  Then, the effective resistance across the join is given by 
$
    R_{\text{tot}} = R_1:R_2 .
$
\end{proposition}
\begin{proof}
  Recall that power with current $\mathbf{I}$ dissipates across a parallel join according to the infimal convolution,
\begin{align}
  P_p(\mathbf{I}) &= 
  \begin{array}{ll}
    \min & \Tr(\mathbf{I}_1^{\intercal}  R_1 \mathbf{I}_1) + \Tr(\mathbf{I}_2^{\intercal}  R_2 \mathbf{I}_2)\\
    \text{s.t.} & \mathbf{I}_1+ \mathbf{I}_2 = \mathbf{I} 
  \end{array} \label{eq:10}\\
         &\triangleq \left(P_1 \square P_2 \right)(\mathbf{I}).
\end{align}
Applying the identity 
$
  (f \square g)^* = f^* + g^*,
$
where $f^*$ denotes the Fenchel conjugate, to Problem~\eqref{eq:10} yields the minimum $P_p(\mathbf{I}) = \Tr(\mathbf{I}^{\intercal}  (R_1:R_2) \mathbf{I})$.
\end{proof}

Furthermore, the $i$th $k\times k$ block on the diagonal of the matrix, 
\begin{align}
  A(W)^{-1} 
           &= \left(\sum_{\{i,j\}\in\Edges} \mathbf{A}_{ij}W_{ij} \mathbf{A}_{ij}^{\intercal}  + \sum_{i\in R} \mathbf{e}_i W_i\mathbf{e}_i^{\intercal}  \right)^{-1}
\end{align}
form the matrix-valued effective resistances from node $i \in \Nodes\setminus \cal R$ to $\cal R$, denoted $R^{\text{eff}}_u(\cal R)$.
Note that when $x \in \mathbb{R}^{nk\times k}$ is a matrix of stacked $k\times k$ current matrices injected into $n$ nodes of $\graph$, then the matrix $A^{-1} x$ is the stacked matrix of the voltage drops from each node to the grounded leader node set.
In particular, the $s$th $k\times k$ block of $A^{-1}x$ is denoted $\blk_s^{k\times k}[A^{-1} x]$, and if $x = (e_s\otimes I_k) = \mathbf{e}_s$, then this corresponds to a current of identity intensity injected into node $s$;
again, this setup is shown in Figure~\ref{fig:leaderFollower}.
Finally, we denote the quantity 
$
  \mathbf{Y}_i^s = \blk_i^{k\times k}[A^{-1} \mathbf{e}_s],
$
as the voltage drop from node $i$ to $\mathcal{R}$ under identity current injected into node $s$;
this reduces to the corresponding scalar definition seen in~\cite{Chapman2015} when $k=1$.

\subsection{Series-Parallel Graph Models}

In this paper, we consider the class of graphs known as \emph{series-parallel graphs}.
Given a series-parallel graph, there exist efficient (i.e., $\mathcal{O}(\log |\Nodes|)$) algorithms that decompose the graph into atomic structures and simple composition operations on them~\cite{Valdes1979,Eppstein1992,He1987}.

\begin{definition}[Two-Terminal Series-Parallel Graphs]
  \label{def:ttsp}
An acyclic graph is called \emph{two-terminal series-parallel (TTSP)} if  it can be defined recursively as follows:
\begin{enumerate}
  \item The graph defined by two vertices connected by an edge (a \emph{1-path}) is a TTSP graph, where one node is labeled the \emph{source}, and the other the \emph{sink}.
  \item If $\graph_1=(\Nodes_1,~\Edges_1)$ and $\graph_2=(\Nodes_2,~\Edges_2)$ are TTSP where $\mathcal{S}_i = \{s_i\},\mathcal{T}_i=\{t_i\}$ are the unique source and sink of $\graph_i$, then the following operations produce TTSP graphs: 
    \begin{enumerate}
      \item \textbf{Parallel Addition:} $\graph_p \leftarrow s_1\sim s_2,~t_1\sim t_2$.
      \item \textbf{Series Addition:} $\graph_s \leftarrow t_1\sim s_2  $.
    \end{enumerate}
\end{enumerate}  
Denote the parallel join of $\graph_1$ and $\graph_2$ as $\graph_1\oslash\graph_2$, and the corresponding series join as $\graph_1\odot\graph_2$, see Fig~\ref{fig:simplejoin}.
\end{definition}

\begin{figure}
  \centering
  \includegraphics[width=0.6\columnwidth]{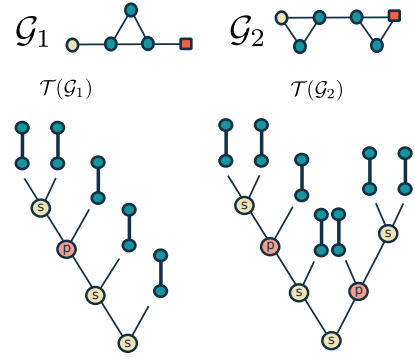}
  \caption{Decomposition trees of two graphs.}
  \label{fig:decompTree}
\end{figure}

The two recursive operations defining the TTSP graph model allow for a simple constructive approach for defining graphs from atomic elements.
Indeed, efficient algorithms exist that decompose TTSP graphs into a \emph{decomposition tree} with the following structure~\cite{Valdes1979,Eppstein1992,He1987}. 
\begin{definition}[TTSP Decomposition Tree]
  A \emph{TTSP decomposition tree} of a TTSP graph $\graph$ is a binary tree $\mathcal{T}(\graph)$ with the following properties: 
  \begin{enumerate}
    \item $\mathcal{T}$ is a complete (sometimes called \emph{full}) binary tree, in that every node has either 2 or 0 children.
    \item Every leaf of $\mathcal{T}$ corresponds to a 1-path.
    \item Every parent of $\mathcal{T}$ corresponds to either a series or parallel addition operation from Def.~\ref{def:ttsp} on its children.
  \end{enumerate}
\end{definition}
In the following proposition, we quantify the  height of the tree in terms of the size of the resulting graph.

\begin{proposition}[Properties of $\mathcal{T}(\graph)$~\cite{Hudobadebadyn2019}]
  \label{prop:tree}
  Let $\graph$ be a TTSP graph with $N$ nodes constructed from $l$ 1-paths with $p$ parallel joins and $s$ series joins.
  Then, the heignt of $\mathcal{T}(\mathcal{G})$ is bounded by
 $
        \log_2\left(N+2p-s\right) \leq h \leq \frac{1}{2}({N+2p+s} -2).
$

\end{proposition}


%% file: 4-performance.tex
\section{System-Theoretic Analysis on Series-Parallel Graphs}
\label{sec:syst-theor-comp}

\subsection{Synthesis of $\Htwo$-Optimal Networks}

In this section, we discuss an efficient algorithm on series-parallel networks for an \emph{a priori} synthesis computation of the $\Htwo$ norm.
First, we define an \emph{all-input series-parallel graph (AITTSP)}, an example of which is shown in Figure~\ref{fig:examples}.
\begin{definition}[All-Input TTSP Graphs]
  Consider a graph $\graph$ in the setup of the leader-follower consensus dynamics.
  Identify each leader node $i_1,i_2,\dots,i_r\in \mathcal{R}$ as one node $l$ connected to all source nodes $s\in R$, as depicted in Figure~\ref{fig:leaderFollower}.
  The graph $\graph$ is called an \emph{all-input two-terminal series-parallel graph} if, for all source nodes $s\in R$, $\graph$ is TTSP with $s$ as the source and $l$ as the sink.
\end{definition}

Next, let us examine the structure of the $\Htwo$ norm in the context of the leader-follower consensus setup in \S\ref{sec:problem-statement}.
\begin{lemma}
\label{lem:4}
  Consider the leader-follower consensus setup 
  with $C=I\otimes I_k$.
  Then, 
  \begin{align}
    \left(\Htwo^{\graph, B}\right)^2 
                                     &= \dfrac{1}{2}\sum_{s\in R} \Tr\left[\mathbf{Y}_s^s\right] 
                                      = \dfrac{1}{2}\sum_{s\in R} \Tr\left[\rho^{\mathrm{eff}}_{s,\mathcal{R}}\right].
  \end{align}
\end{lemma}
\begin{proof}
 We note that,
  \begin{align}
     \left(\Htwo^{\graph, B}\right)^2 &= \dfrac{1}{2} \Tr\left(\mathbf{B}^{\intercal}  A^{-1} \mathbf{B}\right)\\
                                      &= \dfrac{1}{2} \Tr\left(\sum_{i,s \in R} (e_i\otimes I_k) A^{-1}_{is} (e_s\otimes I_k)^{\intercal}  \right)\\
                                      &= \dfrac{1}{2} \Tr\left(\sum_{s\in R} (e_s\otimes I_k) A^{-1}_{is} (e_s\otimes I_k)^{\intercal} \right)\\
                                      &= \dfrac{1}{2} \sum_{s\in R}  \Tr\left[\blk_s^{k\times k}\left[A^{-1} \mathbf{e}_s\right]\right]\\
                                      &= \dfrac{1}{2} \sum_{s\in R} \Tr\left[\mathbf{Y}_s^s\right] 
                                      = \dfrac{1}{2} \sum_{s\in R} \Tr\left[\rho^{\text{eff}}_{s,\mathcal{R}}\right]\label{eq:13}.
  \end{align}
\end{proof}
Each quantity $\mathbf{Y}_s^s$ on the left side of the sum in~\eqref{eq:13} is the voltage drop from the source node $s\in R$ to the grounded leader node set $\mathcal{R}$ (depicted in Figure~\ref{fig:leaderFollower}: Right, if $v_s = v_i$).
This is precisely the voltage dropped over the last parallel join of the series-parallel decomposition of an all-input TTSP graph; the join in question is exactly the one depicted in Figure~\ref{fig:examples}.
We can utilize this  observation to efficiently compute the $\Htwo$ norm \emph{a priori} knowing only the weights of the edges and the decomposition tree of $\graph$.

We use the following setup.
Consider a TTSP graph $\graph$ with source node $s$ and sink node $t$.
Ground the source node $s$, and consider the grounded Laplacian $A$ with respect to the grounded source $s$.
This is a leader-follower system with a single leader.

This parallel join, depicted in Figure~\ref{fig:examples}, effectively makes one of the terminals of the resulting graph an element of $\mathcal{R}$, and the other terminal (the sink) an element of $R$.
The control matrix of the leader-follower consensus problem corresponds to exactly those elements in $R$, which are the `sinks' of the TTSP graph used in that computation.
Therefore, our choice of $\mathbf{B}$ must always select the state of the sink vector $t$; hence $\mathbf{B} = \mathbf{e}_t = e_t\otimes I_k$.

We now proceed in two steps.
First, we need a lemma that essentially states that for an arbitrary TTSP graph, there exists an equivalent 1-path TTSP graph with the same effective resistance.
Then, any composition rule on arbitrary TTSP graphs can be reduced to a composition on the equivalent 1-paths, simplifying the analysis.
Afterward, we show that the series-parallel composition of graphs produces a similar series-parallel computation of the $\Htwo$ performance.

\begin{lemma}
  \label{lem:5}
  Consider two graphs: an arbitrary TTSP graph $\graph_1$ with source $s_1$ and sink $t_1$ with effective resistance $\rho^{\text{eff}}_{s_1,t_1}$, and a 1-path TTSP graph $\graph_2$ with source $s_2$ and sink $t_2$ with effective resistance $\rho^{\text{eff}}_{s_2,t_2}$.
  Let their respective control matrices be $\mathbf{B}_1 = \mathbf{e}_{t_1}$ and $\mathbf{B}_2=\mathbf{e}_{t_2}$.
  Further suppose that 
$
    \rho^{\text{eff}}_{s_1,t_1}=\rho^{\text{eff}}_{s_2,t_2}.
$
  Then, 
$
    (\Htwo^1)^2 = (\Htwo^2)^2.
$
\end{lemma}
\begin{proof}
  The setup of the graphs in the statement of the lemma is a leader-follower consensus with grounded (leader) nodes $s_1,s_2$.
  Therefore, we can invoke Lemma~\ref{lem:4}: denoting the graphs' respective Dirichlet Laplacians as $A_1,A_2$ we can compute,
  \begin{align}
     &(\Htwo^1)^2 
                 = \frac{1}{2}\Tr\left[(e_{t_1}^{\intercal}  \otimes I_k)\left[A_1\right]^{-1}(e_{t_1}\otimes I_k)\right]\\
                 &= \frac{1}{2} \rho^{\text{eff}}_{s_1,t_1}
                 = \frac{1}{2} \rho^{\text{eff}}_{s_2,t_2}
                 = \frac{1}{2}\Tr\left[\mathbf{B}_2^{\intercal} \left[A_2\right]^{-1}\mathbf{B}_2\right]
                 =  (\Htwo^2)^2.
  \end{align}
\end{proof}
Lemma~\ref{lem:5} will allow us to reduce the computation of the $\Htwo$ norm of a composite TTSP graph to the computation of $\Htwo$ norms of an equivalent 1-path. 
This allows us to prove the following theorem.
\begin{theorem}
  \label{thr:3}
  Consider the leader-follower consensus setup in \S\ref{sec:problem-statement} on two graphs: an arbitrary TTSP graph $\graph_1$ with source $s_1$ and sink $t_1$, and a second arbitrary TTSP graph $\graph_2$ with source $s_2$ and sink $t_2$. 
  Let the Dirichlet Laplacian of $\graph_i$ grounded with respect to its source $s_i$ be $A_i$, and its control matrix be $\mathbf{B}_i = \mathbf{e}_{t_i} = e_i\otimes I_k$.
  Then, the $\Htwo$ norm of the corresponding dynamics is given by 
 $
    (\Htwo^{\graph_i})^2 = \frac{1}{2}\Tr[\mathbf{B}_i^{\intercal}  \left[A_i\right]^{-1} \mathbf{B}_i ].
 $
  Furthermore, 
  \begin{align}
    \left(\Htwo^{\graph_1\odot\graph_2}\right)^2   &= \left(\Htwo^{\graph_1}\right)^2 + \left(\Htwo^{\graph_2}\right)^2\label{eq:15} \\
    \left(\Htwo^{\graph_1\oslash\graph_2}\right)^2 &\leq \left(\Htwo^{\graph_1}\right)^2 :\left(\Htwo^{\graph_2}\right)^2 ,\label{eq:14}
  \end{align}
  with equality in the last display if and only if $\rho^\text{eff}_{s_1,t_1} = c \rho^\text{eff}_{s_2,t_2}$ for $c>0$.
\end{theorem}

\begin{figure} 
    \centering
  \subfloat[Series join of 1-paths]{%
       \includegraphics[width=0.35\columnwidth]{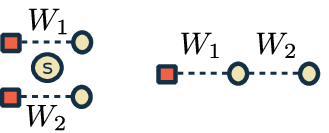}}
    \label{join1}~~
  \subfloat[Parallel join of 1-paths]{%
        \includegraphics[width=0.35\columnwidth]{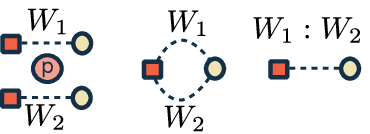}}
    \label{join2}
  \caption{Series and parallel joins of weighted 1-paths.}
  \label{fig:simplejoin}
\end{figure}

\begin{proof}
  By Lemma~\ref{lem:4}, the $\Htwo$ norm of an arbitrary graph can be computed from an equivalent 1-path.
  Hence, we need to show~\eqref{eq:15} and~\eqref{eq:14} for series and parallel joins of 1-paths.
  
  Consider the 1-paths in Figure~\ref{fig:simplejoin}, with weights $W_1,W_2 \in \mathcal{S}^k_{++}$ between the sources (square nodes) and sinks (circular nodes).
  The Dirichlet Laplacians of both with respect to the grounded source nodes are simply 
$
    \lap_{\graph_1} = \left[ W_1 \right],~\lap_{\graph_2} = \left[ W_2 \right],
$
  and their respective control matrices are $B_{\graph_1}= B_{\graph_2} = I_k$
  and $\Htwo$ norms are 
 $
    (\Htwo^{\graph_1})^2 = \frac{1}{2}\Tr [W_1^{-1}]$, $(\Htwo^{\graph_2})^2 = \frac{1}{2} \Tr[W_2^{-1}].
 $

  Similarly, the Laplacians of the series and parallel joins in Figure~\ref{fig:simplejoin} are,
  \begin{align}
     \lap_{\graph_1 \oslash\graph_2} = \left[ W_1+W_2\right] ,~\lap_{\graph_1 \odot\graph_2} = 
    \begin{bmatrix}
      W_1+W_2 & -W_2\\
      -W_2 & W_2
    \end{bmatrix}.
  \end{align}
  The control matrices are 
$
    B_{\graph_1 \oslash\graph_2} = I_k$, $B_{\graph_1 \odot\graph_2} = 
[
      \mathbf{0}_{k\times k} ~ I_k
].
$
  Therefore, the $\Htwo$ norm in the series join case is 
  \begin{align}
   & \left(\Htwo^{\graph_1\odot\graph_2}\right)^2\\   &= \frac{1}{2} \Tr\left[ 
                                                     \begin{bmatrix}
                                                       \mathbf{0}_{k\times k} & I_k
                                                     \end{bmatrix}
                                                           \begin{bmatrix}
                                                             W_1+W_2 & -W_2\\
                                                             -W_2 & W_2      
                                                           \end{bmatrix}^{-1}
                                                     \begin{bmatrix}
                                                       \mathbf{0}_{k\times k} \\ I_k
                                                     \end{bmatrix}\right]\\
                                                   &= \frac{1}{2}\left( W_2 - W_2\left(W_1 + W_2 \right)^{-1}W_2 \right)^{-1},\label{eq:11}
  \end{align}
  where \eqref{eq:11} follows from standard block matrix inversion formulae.
  Note that we can write 
  \begin{align}
   & W_1(W_1+W_2)^{-1}W_2
                         = W_2 - W_2\left(W_1 + W_2 \right)^{-1}W_2.
  \end{align}
  Inverting this expression and substituting into~\eqref{eq:11} yields,
  \begin{align}
    &\left(\Htwo^{\graph_1\odot\graph_2}\right)^2 = \frac{1}{2}\left(W_1(W_1+W_2)^{-1}W_2\right)^{-1} \\
                                                 &= \frac{1}{2}W_1^{-1} + \frac{1}{2}W_2^{-1}
    = \left(\Htwo^{\graph_1}\right)^2 + \left(\Htwo^{\graph_2}\right)^2,
  \end{align}
  as desired.
  We can compute the $\Htwo$ norm in the parallel join case by invoking Theorem 13 from \cite{Anderson1969}, which states that for $A,B \in \mathcal{S}_{++}^k$, 
$
    \Tr(A:B) \leq \Tr(A) : \Tr(B),
$
  with equality if and only if $A = cB$ for some $c>0$.
  By this result, and by Lemma~\ref{lem:4}, we can compute:
  \begin{align}
    \left(\Htwo^{\graph_1\oslash\graph_2}\right)^2 &= \frac{1}{2}\Tr\left[\rho^{\text{eff}}_1 : \rho^{\text{eff}}_2 \right]
                                                   \leq \frac{1}{2} \left(\Tr\left[\rho^{\text{eff}}_1\right] : \Tr\left[ \rho^{\text{eff}}_2 \right]\right) \\
& = \left(\Htwo^{\graph_1}\right)^2 :\left(\Htwo^{\graph_2}\right)^2.
  \end{align}
This concludes the proof.
\end{proof}

We now propose Algorithm~\ref{alg:4} for computing the $\Htwo$ norm of an AITTSP graph with control input $e_s$.
If the matrix weights across the graph are scalar multiples of each other (which is the case when $k=1$), Algorithm~\ref{alg:4} computes precisely the $\Htwo$ norm; otherwise it computes an upper bound.
\removelatexerror
\begin{algorithm}[H]
  \label{alg:4}
  \SetAlgoNoLine
 \caption{$\Htwo$ norm of AITTSP Graph with $\mathbf{B}=\mathbf{e}_i$}
\SetAlgoLined
\KwIn{Decomposition tree $\mathcal{T}(\graph)$, weights $\Weights(\graph)$, $\mathbf{e}_i$}
\KwResult{ Upper bound on $(\Htwo^{\graph})^2$ }
 \For{each leaf ${\mathcal{L}}$ of $\mathcal{T}(\graph)$}{
  Output $(\Htwo^{\mathcal{L}})^2 = \frac{1}{2}W_{\mathcal{L}}^{-1}$ to parent\;
 }
 \For{each parent $j$ of $\mathcal{T}(\graph)$}{
  \eIf{received $\Htwo^{\graph_i}$ from both children}{
    \eIf{$j$ is a series join}{
       Output $(\Htwo)^2 \leftarrow \left(\Htwo^{\graph_1}\right)^2 + \left(\Htwo^{\graph_2}\right)^2 $ to parent\;
    }{
       Output $(\Htwo)^2_{\text{est}} \leftarrow \left(\Htwo^{\graph_1}\right)^2 :\left(\Htwo^{\graph_2}\right)^2 $ to parent\;
    }
   }{
   wait\;
  }
 }
\Return{$(\Htwo)^2$ at root node of $\mathcal{T}(\graph)$.}
\end{algorithm}
\begin{theorem}
  \label{thr:1}
Consider the leader-follower consensus dynamics as in \S\ref{sec:problem-statement} on an all-input TTSP graph $\graph$.
Then, the $\Htwo$ norm of this system is given by 
\begin{align}
  \left(\Htwo^{\graph, \mathbf{B}}\right)^2 \triangleq -\frac{1}{2}\Tr\left[ \mathbf{B}^{\intercal}  A^{-1} \mathbf{B} \right] = \sum_{s\in \mathcal{R}}  \left(\Htwo^{\graph, \mathbf{e}_s}\right)^2,
\end{align}  
and the best-case complexity of computing this $\Htwo$ norm is $\mathcal{O}(k^\omega|\mathcal{R}| \log |\Nodes|)$, and the worst-case complexity is $\mathcal{O}(k^\omega|\mathcal{R}||\Nodes|)$.
\end{theorem}
\begin{proof}
  We can compute: 
  \begin{align}
    \left(\Htwo^{\graph, \mathbf{B}}\right)^2  = - \frac{1}{2} \sum_{s\in\mathcal{R}} \Tr\left[ \mathbf{e}_s^{\intercal}  A^{-1} \mathbf{e}_s \right] = \sum_{s\in\mathcal{R}} \left[ \Htwo^{\graph_2,\mathbf{e}_s} \right]^2.\label{eq:h2}
  \end{align}
  This is a sum of $|\mathcal{R}|$ $\Htwo$ norms of leader-follower consensus networks with control input $\mathbf{e}_s$.
  Since at each layer of the decomposition tree $\mathcal{T}(\graph)$ each computation happens independently, the complexity depends on the height $h$ of the decomposition tree.
  Hence, the complexity is determined by $h$, which by Proposition~\ref{prop:tree} is $\mathcal{O}( \log |\Nodes|)$, and  $\mathcal{O}(|\Nodes|)$ for best and worst-case, respectively.
 Each call of the three algorithms must be done $|\mathcal{R}|$ times.
 Algorithm 1 may perform an inversion of a $k\times k$ matrix, which has complexity $\mathcal{O}(k^\omega)$.
 Algorithm 2 requires computing $(\mathbf{I}_1,\mathbf{I}_2) = \arg P_1\square P_2$, given by
 \begin{align}
   (R_{\text{eff}_1}^{-1} (R_{\text{eff}_1}^{-1}:R_{\text{eff}_2}^{-1}) \mathbf{I}_{\text{in}},R_{\text{eff}_2}^{-1} (R_{\text{eff}_1}^{-1}:R_{\text{eff}_2}^{-1}) \mathbf{I}_{\text{in}}),
 \end{align} 
 over every parallel join, which also requires inverting  two $k\times k$ matrices (the parallel addition of resistances is computed in Algorithm 1).
\end{proof}

\subsection{Noise Rejection and Adaptive Weight Design}

{It is often the case that one wishes to adapt the network in order to reject noise.
For example, in a swarm of UAVs performing consensus on heading, it is undesirable for the swarm to be influenced by wind gusts\cite{Chapman2015}.
Similarly, one may wish to design matrix-weighted networks in order to reject noise during distributed attitude control and estimation~\cite{Trinh2018a}.
In this section, we discuss a protocol that allows the network to quickly adapt its interaction edge weights to minimize the collective network response to such disturbances.}

Consider the leader-follower consensus dynamics setup from 
\S\ref{sec:problem-statement}, and  the task of re-assigning weights to the edges of the network to minimize the $\Htwo$ norm.
This can be done via the optimization problem, 
\begin{align}
  \begin{array}{ll}
    \min &  \left(\Htwo^{\graph, \mathbf{B}}\right)^2 = \frac{1}{2}\Tr\left(\mathbf{B}^{\intercal} A(W)^{-1}\mathbf{B}\right)\\
    \text{s.t.} & W := \mathrm{Blkdiag}(W_e)\\
         &W_e\in \mathcal{S}^{k\times k}_{++},~\forall e\in\Edges\\
     & U_e \succeq W_e \succeq L_e,~\forall e\in\Edges,~L_e,P_e\in\mathcal{S}_{++}^{k\times k}.
  \end{array}\label{eq:6}
\end{align}

\begin{remark}
  \label{rem:1}
  One must include the bounds in Problem~\eqref{eq:6} in order to prevent edges from becoming disconnected, or from becoming arbitrarily `large' (in the sense of the Loewner order on $W_e$).
This also motivates adding a  regularizer term $\frac{1}{2} h \sum_{e\in \Edges}\|W_e\|^2$ for some matrix norm $\|\cdot\|$ to the cost function. 
\end{remark}

The key features of optimization problem~\eqref{eq:6} are highlighted in the following proposition and proven in the appendix.
\begin{proposition}
\label{prop:optim}
  Consider the setting of Problem~\eqref{eq:6}.
  Then, 
  \begin{enumerate}
    \item The objective  $f_W = \frac{1}{2}\Tr\left(\mathbf{B}^{\intercal} A(W)^{-1}\mathbf{B}\right)$ is strongly convex on the cone of positive-definite matrices $W_e\in\mathcal{S}_{++}^k$.
    \item The gradient of the objective function with respect to a single edge weight $W_e$ at a point $H\in \mathcal{S}_{++}^k$ is given by $\nabla f_{W_e}[H] = -\frac{1}{2}Q Q^{\intercal}$, where 
$
        Q = \mathbf{A}_e^{\intercal}  A^{-1} \mathbf{e}_s,
$
and
      $\mathbf{Y}_i^s = \blk_i^{k\times k} \left[ A^{-1} \left( e_s\otimes I_k \right) \right]$, where $A^{-1}$ is evaluated with $H$ in place of $W_e$.
  \end{enumerate}
\end{proposition}

We can solve Problem~\eqref{eq:6} using a projected gradient descent algorithm:
\begin{align}
x^{t+1} &= \mathrm{Proj}_{\mathcal{C}}\left( x^t + \dfrac{1}{h\sqrt{t}} \nabla_x f(x^t) \right),
\end{align}
where $\mathcal{C}$ is the cone generated by the constraints of Problem~\eqref{eq:6}.

Following Remark~\ref{rem:1}, let us include a Frobenius norm penalty on the edge weights in the cost: 
\begin{align}
  f(W) &= \frac{1}{2}\Tr\left[\mathbf{B}^{\intercal}  A^{-1}\mathbf{B}\right] + \frac{h}{2}\sum_{e\in\Edges}\Tr\left(W_e^{\intercal} W_e\right)^2.
\end{align}
The gradient of the penalty term with respect to $W_e$ is simply $hW_e$, so
each edge in the network updates its weight according to the dynamics given by the gradient update
\begin{align}
  W_{e}^{t+1} =\mathrm{Proj}_{\mathcal{C}}\left[ \left(1-\frac{1}{\sqrt{t}}\right)W_{e}^t + \dfrac{1}{h\sqrt{t}} \nabla f_{W_{e}}[W_{e}^t]\right].\label{eq:8}
\end{align}

Problem~(\ref{eq:6}) in the scalar edge weight case was examined in \cite{Chapman2015}.
In this case, an efficient projection onto the constraint set is via the $\|\cdot\|_\infty$ norm.
In the case of matrix edge weights, there is no such natural projection.
Instead, at each step one solves the problem 
$
\mathrm{Proj}_{\mathcal{C}}(X) = \min_{Y\in {\mathcal {C}} \cap \mathcal{S}_{++}^n}  \|Y-X\|
$
for some matrix norm $\|\cdot\|$.
However, if $k$ is relatively small compared to the size of the overall network, then this operation is not the most computationally expensive part of the setup.
The complexity in solving Problem~\eqref{eq:6} lies in computing the gradient $\nabla f_{W_e}[H]$, as it requires the voltage drops $\mathbf{Y}_i^s,\mathbf{Y}_j^s$ for each edge $\{i,j\}\in \Edges$.

We now present an algorithm for computing this quantity on all-input two-terminal series-parallel graphs, and a characterization of its complexity.
Informally, the algorithm is as follows.
For each source node $s$, the algorithm utilizes the decomposition tree of $\graph$ with $s$ as the source and the grounded leader set as the sink; this is why $\graph$ needs to be TTSP with respect to all source nodes.
The voltage drops can be computed from resistances and currents across each join, and the currents can be extracted from the power dissipated across each join.
Hence, the effective resistances are computed first, as in Algorithm~\ref{alg:1}.
Starting from the root of the decomposition tree, the currents at each join can be computed by Eq.~\eqref{eq:10}, as in Algorithm~\ref{alg:2} and depicted in Figure~\ref{fig:computation}.
Finally, the voltage drops over each branch are computed starting from the leaves of the decomposition tree, as in Algorithm~\ref{alg:3}.
\begin{figure}
  \centering
  \includegraphics[width=0.7\columnwidth]{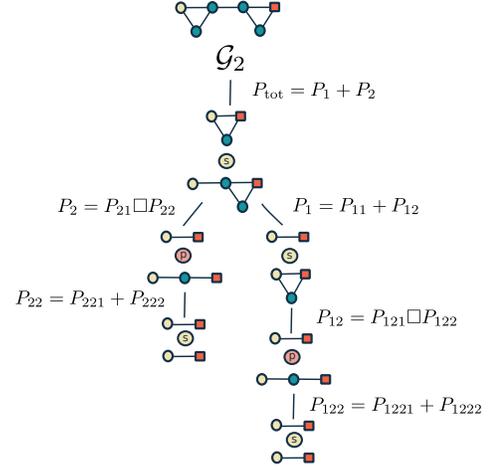}
  \caption{Nested infimal convolution/sum computations over the decomposition tree \& series-parallel graph in Figure~\ref{fig:decompTree}. Beige circles (always on the left) denote sources, red squares (always on the right) denote sinks at each step.}
  \label{fig:computation}
\end{figure}
\removelatexerror
\begin{algorithm}[H]
  \label{alg:1}
  \SetAlgoNoLine
 \caption{Effective Resistance over AITTSP Graph}
\SetAlgoLined
\KwIn{Decomposition tree $\mathcal{T}(\graph)$}
\myinput{Edge weights $\Weights(\graph)$}
\KwResult{ Effective resistances $\rho(\graph)$ over $\mathcal{T}(\graph)$ }
 \For{each leaf of $\mathcal{T}(\graph)$}{
  Output $R_{\text{eff}} = W_e^{-1}$ to parent\;
 }
 \For{each parent $j$ of $\mathcal{T}(\graph)$}{
  \eIf{received $R_{\text{eff}_i}$ from both children $i=1,2$}{
    \eIf{$j$ is a series join}{
       Output $R_{\text{eff}} = R_{\text{eff}_1}+R_{\text{eff}_2} $ to parent\;
    }{
       Output $R_{\text{eff}} = R_{\text{eff}_1}:R_{\text{eff}_2} $ to parent\;
    }
   }{
   wait\;
  }
 }
\end{algorithm}

\removelatexerror
\begin{algorithm}[H]
  \label{alg:2}
 \caption{Branch Currents over AITTSP Graph}
\SetAlgoLined
\KwIn{Decomposition tree $\mathcal{T}(\graph)$}
\myinput{Effective resistances  $\rho(\graph)$}
\KwResult{ Currents $\mathcal{I}(\graph)$ over $\mathcal{T}(\graph)$ }
 \For{each parent $j$ of $\mathcal{T}(\graph)$}{
   \uIf{$j$ is root}{
         \eIf{$j$ is a series join}{
           Output $\mathbf{I}_{\text{out}} = I_k$ to children\;
         }{
           Output $(\mathbf{I}_1,\mathbf{I}_2) = \arg P_1\square P_2$ to children\;
         }
   }
   \uElseIf{received $\mathbf{I}_{\text{in}}$ from parent}{
         \eIf{$j$ is a series join}{
           Output $\mathbf{I}_{\text{out}} = \mathbf{I}_{\text{in}}$ to children\;
         }{
           Output $(\mathbf{I}_1,\mathbf{I}_2) = \arg P_1\square P_2$ to children\;
         }
   }
   \Else{
     wait\;
   }
 }

\end{algorithm}
\removelatexerror
\begin{algorithm}[H]
  \label{alg:3}
 \caption{Voltage Drops over AITTSP Graph}
\SetAlgoLined
\KwIn{Decomposition tree $\mathcal{T}(\graph)$}
\myinput{Effective resistances $\rho(\graph)$ over $\mathcal{T}(\graph)$}
\myinput{Currents $\mathcal{I}(\graph)$ over $\mathcal{T}(\graph)$}
\KwResult{Voltage drops $\mathbf{Y}_i^s$ over $\mathcal{T}(\graph)$ }
 \For{each leaf of $\mathcal{T}(\graph)$}{
  Output $\mathbf{V}_e = W_e^{-1}\mathbf{I}_e$ to parent\;
 }
 \For{each parent $j$ of $\mathcal{T}(\graph)$}{
  \eIf{received $\mathbf{V}_{e_i}$ from both children $i=1,2$}{
    \eIf{$j$ is a series join}{
       Output $\mathbf{V}_{\text{out}} = \mathbf{V}_{\text{in}_1}+\mathbf{V}_{\text{in}_2} $ to parent\;
    }{
       Output $\mathbf{V}_{\text{out}} = \mathbf{V}_{\text{in}_1} = \mathbf{V}_{\text{in}_2} $ to parent\;
    }
   }{
   wait\;
  }
 }
\end{algorithm}

\begin{figure} 
  \centering
  \includegraphics[width=0.7\columnwidth]{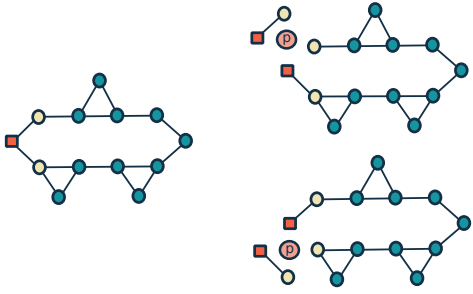}
  \caption{Left: An all-input TTSP graph. Right: Parallel joins across $R$ to $\mathcal{R}$ used to compute $\mathbf{Y}_s^s$.}
  \label{fig:examples} 
\end{figure}

\begin{remark}
Following Theorem~\ref{thr:1}, the  best/worst-case complexity of computing the voltage drops $y_i^s$ in the update scheme of \eqref{eq:8} is $\mathcal{O}(|\mathcal{R}| \log |\Nodes|)$, and  $\mathcal{O}(|\mathcal{R}||\Nodes|)$, respectively.
\end{remark}

\section{Example}
\label{sec:example}
We perform the gradient descent on the TTSP graph in Figure~\ref{fig:eg} with 3 leader nodes, identified to a single node depicted as an orange square.
The weights are assumed to be elements of $\mathcal{S}_{++}^2$, which have 3 independent elements.
These independent elements of each weight $W_{\{i,j\}}$ are represented as a multi-edge, with 3 edges between each node $i$ and its neighbour $j$.
The edges connecting the leader nodes to the source nodes are identity weights, and the remaining weights and bounds were randomly initialized, with weight penalty of $h=0.05$.
The algorithm convergence is shown in Figure~\ref{fig:h2plot}.

{
  In this setup, each node in $\mathcal{G}$ computes one of the independent calculations at each layer of the decomposition tree $\mathcal{T}$.
  This is possible as each layer has at most $\frac{l}{2}$ computations, and $\mathcal{G}$ has at least $\frac{l}{2}+2$ nodes, where $l$ is the number of leaves in the decomposition tree.
  Algorithms~\ref{alg:1}-\ref{alg:3} are then executed in this manner for each $s\in R$, ultimately outputting the voltage drops $\mathbf{Y}_i^s,\mathbf{Y}_j^s$ which are used to compute $\nabla_{W_{\{i,j\}}}f_{W_{\{i.j\}}}$.
  Each edge $\{i,j\}$ is then updated according to~\eqref{eq:8}.
}

\begin{figure}
  \centering
  \includegraphics[width=\columnwidth]{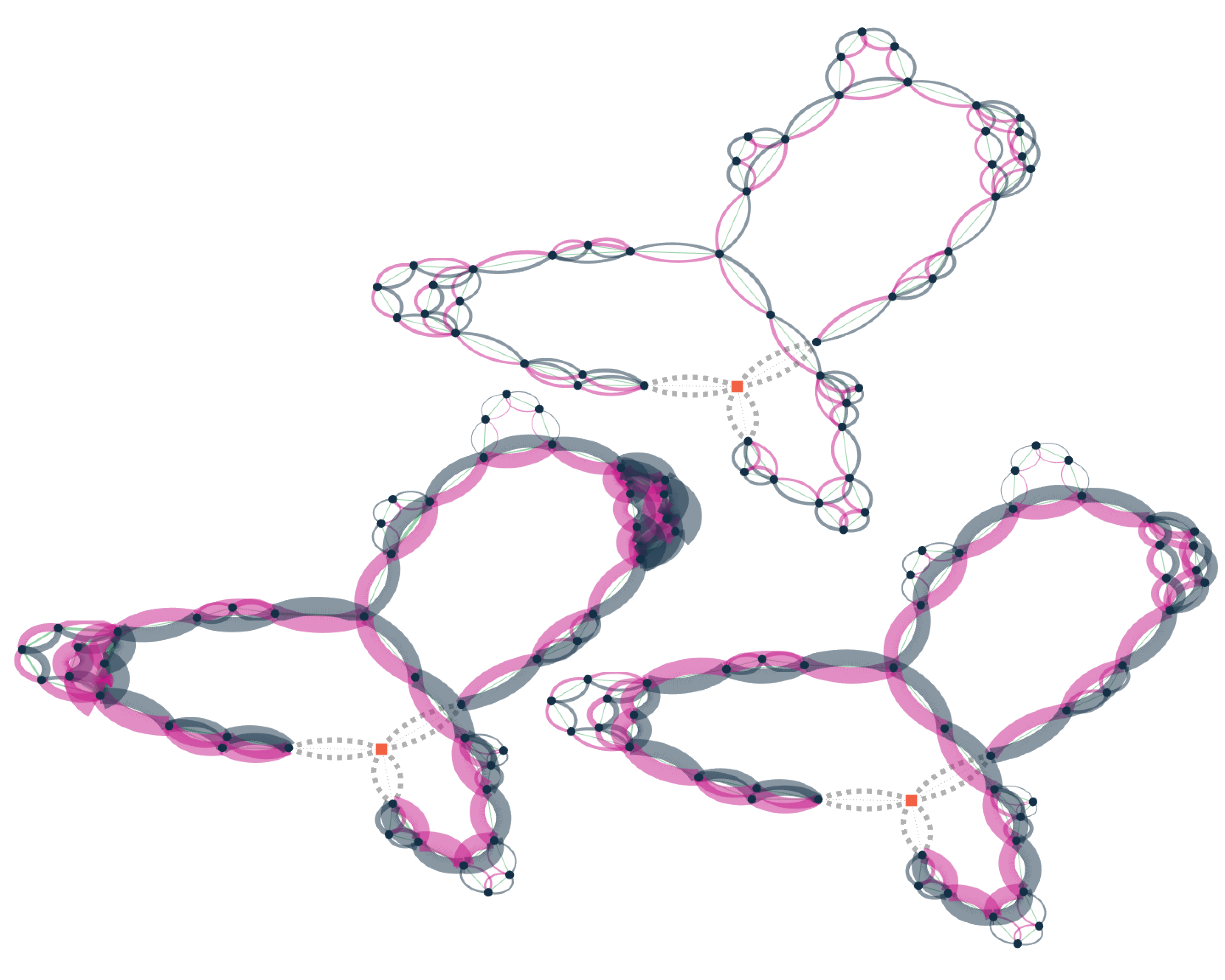}
  \caption{Edge weights at iterations 1 (top), 3 (left), \& 5 (right). Grounded leader node is the orange square. Initial weights exaggerated for clarity.}
  \label{fig:eg}
\end{figure}

\begin{figure}
  \centering
  \includegraphics[width=.9\columnwidth]{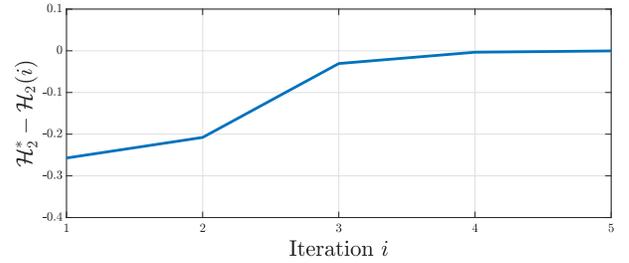}
  \caption{Convergence of gradient descent weight update for optimal $\mathcal{H}_2$ performance on the graphs in Figure~\ref{fig:eg}.}
  \label{fig:h2plot}
\end{figure}

\section{Conclusion}
\label{sec:conclusion}
In this paper we discussed how to utilize the compositional aspects of two-terminal series-parallel graphs to efficiently  compute performance measures on leader-follower consensus networks.
In particular, we have shown that given a decomposition tree of a TTSP graph, one can compute the $\Htwo$ norm, and a gradient update for network reweighting, in best-case and worst-case complexity of $\mathcal{O}(k^\omega|\mathcal{R}| \log |\Nodes|)$ and $\mathcal{O}(k^\omega|\mathcal{R}||\Nodes|)$, respectively.

The authors thank Jingjing Bu and Airlie Chapman for useful conversations.


%% file: 5-appendix.tex
\section*{Appendix}

\textbf{Proof of Proposition~\ref{prop:optim}:} 
In this setup, $ \Htwo^2 $ is given by
\begin{align}
 \dfrac{1}{2} \sum_{s\in R} \Tr \left[ \mathbf{e}_s^\intercal A^{-1} \mathbf{e}_s \right],~
  A 
   = \mathbf{A}_c W_c \mathbf{A}_c^\intercal + \sum_{e\in\Edges \setminus c} \mathbf{A}_e W_e \mathbf{A}_e^\intercal.
\end{align}
We want to find the gradient of the $\Htwo$ norm with respect to $W_c$ for each $c\in\Edges$, and we do this by computing the gradients for each $s \in R$ of $\Tr \left[ \mathbf{e}_s^\intercal A^{-1} \mathbf{e}_s \right]:=(\alpha \circ \beta \circ \gamma) [W_c]$, where,
\begin{align}
  &\alpha(X) = \Tr(\mathbf{e}_s^\intercal X \mathbf{e}_s),~d_{\alpha(X)} [H] = \Tr(\mathbf{e}_s^\intercal H \mathbf{e}_s)\\
  &\beta(X) = X^{-1},~d_{\beta(X)} [H] =  - X^{-1} H X^{-1}\\
  &\gamma(X) = \mathbf{A}_c X \mathbf{A}_c^\intercal + \sum_{e\in\Edges \setminus c} \mathbf{A}_e W_e \mathbf{A}_e^\intercal,~d_{\gamma(X)} [H] = \mathbf{A}_c H \mathbf{A}_c^\intercal\\
  &d(\alpha\circ \beta\circ \gamma) [H] =- \Tr \left[ \mathbf{e}_s^\intercal A^{-1} \mathbf{A}_c H \mathbf{A}_c^\intercal A^{-1} \mathbf{e}_s \right],~A = \gamma(H).
\end{align}
Let $Q = \mathbf{A}_c^\intercal A^{-1} \mathbf{e}_s$.
Then, the gradient at a point $H$ is
\begin{align}
  d(\alpha\circ \beta\circ \gamma) [H] =- \Tr \left[ Q^\intercal HQ \right] = \langle -QQ^\intercal , H \rangle,
\end{align}
and so the gradient at $H$ is $-QQ^\intercal $.
Let  $l = \{i,j\}$. Then, 
\begin{align}
  Q &= \mathbf{A}_l^\intercal A^{-1} \mathbf{e}_s
    = \left(e_i^\intercal \otimes I_k - e_j^\intercal \otimes I_k  \right) A^{-1} \left(e_s \otimes I_k\right).
\end{align}
This is precisely $Q= \blk_i^{k\times k}\left[A^{-1}e_s \right] - \blk_j^{k\times k}\left[A^{-1}e_s \right] := \mathbf{Y}_i^s - \mathbf{Y}_j^s $, or the difference in matrix-valued voltage drops to the grounded leader node from node $i$ to node $j$.
Hence, the gradient with respect to weight $W_c$ is given by 
$
 \nabla_{W_{c}} \Htwo^2[H] = - \frac{1}{2}\sum_{s\in R}\left( \mathbf{Y}_i^s - \mathbf{Y}_j^s \right)\left( \mathbf{Y}_i^s - \mathbf{Y}_j^s \right)^\intercal,
$
where $A^{-1}$ is computed with $H$ in place of $W_{c}$.
